\documentclass[12pt]{amsart}
\makeatletter
\@namedef{subjclassname@2020}{\textup{2020} Mathematics Subject Classification}
\makeatother
\usepackage[utf8]{inputenc}

\usepackage{amsmath}
\usepackage{amsfonts}
\usepackage{amssymb}
\usepackage{amsthm}
\usepackage[shortlabels]{enumitem}
\usepackage{epigraph}

\newcommand{\st}[1][ ]{\ #1|\ }
\newcommand{\ceil}[1]{{\left\lceil #1 \right\rceil}}
\newcommand{\floor}[1]{{\left\lfloor #1 \right\rfloor}}
\newcommand*{\defeq}{\mathrel{\vcenter{\baselineskip0.5ex \lineskiplimit0pt
                     \hbox{\scriptsize.}\hbox{\scriptsize.}}}%
                     =}
\newcommand*{\eqdef}{=\mathrel{\vcenter{\baselineskip0.5ex \lineskiplimit0pt
                     \hbox{\scriptsize.}\hbox{\scriptsize.}}}}

\newcommand{\fdfrac}[2]{\mbox{\footnotesize$\displaystyle\frac{#1}{#2}$}}
\def\eps{\varepsilon}
\def\NN{{\mathbb N}}
\def\ZZ{{\mathbb Z}}
\def\RR{{\mathbb R}}
\def\HH{{\mathbb H}}
\DeclareMathOperator{\id}{id}
\DeclareMathOperator{\diam}{diam}
\DeclareMathOperator{\cay}{Cay}
\DeclareMathOperator{\vol}{Vol}

\usepackage{xcolor}
\definecolor{darkred}{RGB}{180,0,0}
\definecolor{darkblue}{RGB}{0,0,180}
\usepackage{hyperref}
\hypersetup{
  unicode = true,
  colorlinks = true,
  urlcolor   = darkblue,
  linkcolor  = darkblue,
  citecolor  = darkred,
  pdftitle={Coarse entropy of metric spaces}
}
\usepackage[initials]{amsrefs}
\usepackage[capitalise]{cleveref}
\usepackage[normalem]{ulem}

\newtheorem{thm}{Theorem}[section]

\newtheorem{lem}[thm]{Lemma}
\crefname{lem}{Lemma}{Lemmata}
\newtheorem{prop}[thm]{Proposition}
\crefname{prop}{Proposition}{Propositions}
\newtheorem{cor}[thm]{Corollary}
\crefname{cor}{Corollary}{Corollaries}
\theoremstyle{definition}
\newtheorem{defi}[thm]{Definition}
\newtheorem{longEx}[thm]{Example}
\crefname{longEx}{Example}{Examples}
\newtheorem*{claim}{Claim}
\theoremstyle{remark}
\newtheorem{longRem}[thm]{Remark}

\title{Coarse entropy of metric spaces}
\author{William Geller}
\address{Department of Mathematical Sciences, Indiana University Indianapolis, 402 N. Blackford
  Street, Indianapolis, IN 46202}
\email{wgeller@iu.edu}
\author{Micha{\l} Misiurewicz}
\address{Department of Mathematical Sciences, Indiana University Indianapolis, 402 N. Blackford
  Street, Indianapolis, IN 46202}
\email{mmisiure@iu.edu}
\author{Damian Sawicki}
\address{KU Leuven, Department of Mathematics,
Celestijnenlaan 200b – box 2400, 3001 Leuven, Belgium}
\urladdr{https://sites.google.com/view/damiansawicki}

\thanks{This work was partially supported by grant number 426602
  from the Simons Foundation to M.M. D.S.\ was supported by the FWO research project G090420N of the Research Foundation Flanders.}


\date{}

\usepackage{graphicx}
\begin{document}
\begin{abstract}
Coarse geometry studies metric spaces on the large scale. The recently introduced notion of coarse entropy is a tool to study dynamics from the coarse point of view. We prove that all isometries of a given metric space have the same coarse entropy and that this value is a coarse invariant. We call this value the coarse entropy of the space and investigate its connections with other properties of the space. We prove that it can only be either zero or infinity, and although for many spaces this dichotomy coincides with the subexponential--exponential growth dichotomy, there is no relation between coarse entropy and volume growth more generally. We completely characterise this dichotomy for spaces with bounded geometry and for quasi-geodesic spaces. As an application, we provide an example where coarse entropy yields an obstruction for a coarse embedding, where such an embedding is not precluded by considerations of volume growth.
\end{abstract}

\maketitle

\section{Introduction}

In~\cite{GM}, for a metric space $X$ and a map $f\colon X\to X$, the two
first-named authors defined the \emph{coarse entropy} of $f$. The
definition mimics one of the standard definitions of topological
entropy for continuous maps of compact metric spaces. However, it is
adapted to the ideas of coarse geometry, so instead of observing our
dynamical system through a better and better microscope, we use better
and better binoculars, looking from the other end.
Moreover, in coarse geometry one observes only approximate
positions of points, so instead of orbits we consider approximate
orbits, or \emph{pseudoorbits}. We measure the exponential growth
(in the length of the pseudoorbit) of the number of pseudoorbits
that we can distinguish.

In the study of a metric space, its group of isometries is often instrumental. We prove that for a fixed metric space $X$, the coarse entropies of all isometries of $X$ are equal, and in fact the same holds for all isometric embeddings $X\to X$. We call this value the coarse entropy of $X$. Furthermore, we prove that the coarse entropy of a metric space is a coarse invariant, that is, if $X$ and $Y$ are coarsely equivalent, then their coarse entropies must coincide.

As a consequence, coarse entropy can serve as an obstruction for coarse embeddings. For example, \cref{second-example} provides an infinite graph that might be considered as small as possible because its volume growth is linear. On the other hand, \cref{unboundedZero} provides a weighted graph that does not have bounded geometry, so it might be considered large. Interestingly, we prove that the coarse entropy is infinite for the former and zero for the latter, so the former does not admit a coarse embedding into the latter.

An advantage of coarse entropy is that it is defined for an arbitrary metric space.  In particular, we do not need a measure in our definition (cf.~\cite{HMT}). Furthermore, as opposed to some other invariants of metric spaces like Hilbert-space compression, coarse entropy is not merely an invariant of quasi-isometry, but an invariant of coarse equivalence.

While the topological entropy of an isometry of a compact metric space always vanishes, as we have already mentioned the coarse entropy of an unbounded metric space may be infinite. In fact, in \cref{sec:entro} we prove the dichotomy that the coarse entropy of a metric space can only be either zero or infinity.

We investigate connections between the value of the coarse entropy of the
space and other properties of the space. 
In \cref{sec:metric}, we introduce fundamental definitions of metric geometry and prove auxiliary results. 
In \cref{sec:examples}, we provide our main examples, demonstrating that spaces of exponential growth may have vanishing coarse entropy (\cref{s:ze}), spaces without bounded geometry may have vanishing coarse entropy as well (\cref{s:zu}), while spaces of linear volume growth may have infinite coarse entropy (\cref{s:is}).

Despite the existence of such examples, in~\cref{sec:theorems} we manage to prove theorems establishing the value of coarse entropy for large classes of spaces.
In particular, in \cref{sec:boundedGeometry} we completely characterise the dichotomy between zero and infinite coarse entropy for metric spaces with bounded geometry. In~\cref{sec:unboundedGeometry}, we treat metric spaces without bounded geometry
at the price of the additional assumption that they are quasi-geodesic.
As a corollary, we obtain a complete characterisation of coarse entropy for all quasi-geodesic spaces (with and without bounded geometry).

\section{The coarse entropy of a metric space}\label{sec:entro}

Let us start with the definition of the coarse entropy \cite{GM} of a map
$f\colon X\to X$, where $(X,d)$ is a metric space. It is defined as
\begin{equation}\label{firstInThePaper}
h_\infty(f)=\lim_{\delta\to\infty} \lim_{R\to\infty}
\limsup_{n\to\infty} \fdfrac1n \log s(f,n,R,\delta,x_0),
\end{equation}
where $s(f,n,R,\delta,x_0)$ is the supremum of cardinalities of
$R$-separated subsets of the set $P(f,n,\delta,x_0)$ of \emph{$\delta$-pseudo\-orbits} of $f$ of length $n$
starting at $x_0$. As usual, a $\delta$-pseudoorbit of $f$ of length $n$
starting at $x_0$ is a sequence $(x_0,x_1,\dots,x_n)$ such that
$d(f(x_i),x_{i+1})\le\delta$ for $i=0,1,\dots,n-1$. A set is \emph{$R$-separated} if the distance between each
two distinct elements of this set is at least $R$, and the distance
between two pseudo\-orbits $(x_0,x_1,\dots,x_n)$ and
$(y_0,y_1,\dots,y_n)$ is the maximum of the distances $d(x_i,y_i)$ over
$i=0,1,\dots,n$.
The value of $h_\infty(f)$ in the above definition does not depend on the
choice of $x_0\in X$.

We will always assume that isometries are bijective. Not necessarily surjective distance-preserving maps will be called \emph{isometric embeddings}.

If for sequences $u = (z_0, \ldots z_n)$ and $u' = (z'_0, \ldots z'_{n'})$ we have $z_n = z'_0$, then by $u * u'$ we will denote their concatenation $(z_0, \ldots z_n, z'_1, \ldots z'_{n'})$.

\begin{thm}\label{isomEmbedding}
For a metric space $X$, the coarse entropy of any isometric embedding $f\colon X\to X$ is the same.
\end{thm}
\begin{proof}
Note that the value of $s(f,n,R,\delta,x_0)$ depends only on $R$ and the metric properties of the set $P(f,n,\delta,x_0)$. Let $f\colon X\to X$ be any isometric embedding and $\id_X\colon X\to X$ be the identity map and consider the map $\widetilde f \colon P(\id_X, n,\delta,x_0) \to P(f,n,\delta,x_0)$ given by 
\begin{equation*}
\widetilde f(x_0, \ldots, x_n) = (x_0, f(x_1), \ldots, f^n(x_n)).
\end{equation*}
The map $\widetilde f$ indeed maps $\delta$-pseudoorbits of $\id_X$ to $\delta$-pseudoorbits of $f$ because for any $i=0,\ldots, n-1$ we have
\[d(x_i, x_{i+1})\leq \delta \iff d\left(f\left(f^i(x_i)\right), f^{i+1}(x_{i+1})\right)\leq \delta,\] 
and  $\widetilde f$ is an isometric embedding because
\[\max_{i\leq n} d(x_i, y_i) = \max_{i\leq n} d\left(f^i(x_i), f^i(y_i)\right).\]
Hence, $P(\id_X, n,\delta,x_0)$ embeds isometrically into $P(f,n,\delta,x_0)$, and thus $s(\id_X, n,R,\delta,x_0)$ is at most $s(f,n,R,\delta,x_0)$, so $h_\infty(\id_X)\leq h_\infty(f)$.
(It can be observed that when $f$ is an isometry, then $\widetilde f$ is an isometry, which yields the desired equality $h_\infty(\id_X)=h_\infty(f)$ in this special case.)

It remains to show that $h_\infty(\id_X) \geq h_\infty(f)$. If $h_\infty(f) = 0$, there is nothing to prove, so let us assume that $h_\infty(f) > 0$. Let $C>1$ be such that $\log C < h_\infty(f)$. There is $\delta > 0$ such that for every $R > 0$ there is a sequence $(n_k)$ such that $P(f, n_k, \delta, x_0)$ contains an $R$-separated subset $S_k$ of cardinality at least $C^{n_k}$. Consider the map $F\colon P(f,n_k,\delta,x_0) \to P(\id_X, n_k,\delta,f^{n_k}(x_0))$ given by
\[F(x_0, \ldots, x_n) = (f^{n_k}(x_0), f^{n_k-1}(x_1), \ldots, x_n). \]
Similarly as above, it is easy to see that $F$ is a well-defined isometric embedding.

Now let $D=d(x_0, f(x_0))$. We can assume that $\delta \geq D$. Since $f$ is an isometric embedding, we have $D=d(f^i(x_0), f^{i+1}(x_0))$, and hence $p=(x_0, f(x_0), \ldots, f^{n_k}(x_0))$ is a $D$-pseudoorbit of $\id_X$. Consequently, the set $\{p*q \st q\in F(S_k)\}$ of concatenations of $p$ with $\delta$-pseudoorbits in $F(S_k)$ is a subset of $P(\id_X, 2n_k,\delta,x_0)$. Clearly, it is $R$-separated and has cardinality at least $C^{n_k}$. It follows that $h_\infty (\id_X) \geq \frac1 2 \log C$, and since $C>1$ satisfying $\log C < h_\infty(f)$ was arbitrary, we conclude $h_\infty (\id_X) \geq \frac{1}{2} h_\infty(f)$. By Theorem~3.6 of~\cite{GM} we obtain that $h_\infty(f^2) = 2h_\infty(f) > 0$ because an isometric embedding is in particular a controlled map (in the sense of~\cite{GM}). By applying the inequality from the penultimate sentence to $f^2$ instead of $f$, we conclude $h_\infty(\id_X) \geq \frac{1}{2} h_\infty(f^2) = h_\infty(f)$.
\end{proof}

By \cref{isomEmbedding}, the following is well defined.

\begin{defi}\label{mainDefinition} Let $X$ be a metric space. The \emph{coarse entropy} of $X$, denoted $h_\infty(X)$, equals $h_\infty(f)$, where $f\colon X\to X$ is any isometric embedding.
\end{defi}

Let us start by discussing some examples.
It follows from \cite{GM}*{Lemma~4.2} that the coarse entropy of any finite-dimensional normed vector space over $\RR$ equals zero. Infinite-dimensional vector spaces provide the simplest class of examples of spaces with infinite coarse entropy. We prove that their coarse entropy is infinite in \cref{lt1} using the following lemma.

\begin{lem}\label{ll1}
Let $X$ be a normed vector space over $\RR$ of infinite dimension. Then the unit
sphere in $X$ contains an infinite $1$-separated set.
\end{lem}

\Cref{ll1} is a well-known fact in functional analysis. We provide an elementary proof for convenience.
\begin{proof}[Proof of \cref{ll1}]
We will use induction. Fix $n\ge 0$. Suppose that $E_n$ is a
$1$-separated subset of the unit sphere, consisting of $n$ elements.
Let $H$ be the linear subspace of $X$ spanned by $E_n$. For any
$x\in X$ define a function $\xi_x\colon H\to\RR$ by $\xi_x(y)=\|x-y\|$. Fix
$x\in X\setminus H$. Since $H$ is finite dimensional, closed bounded
subsets of $H$ are compact. Moreover, $\xi_x$ is continuous and
$\xi_x(y)$ goes to infinity as $\|y\|$ goes to infinity. Therefore,
$\xi_x$ attains its infimum at some point $y'\in H$. Set $x'=x-y'$.
Then $\xi_{x'}$ attains its infimum at $0$. Set $x''=x'/\|x'\|$. Then
for every $y\in H$
\[
\xi_{x''}(y)=\frac{\xi_{x'}(\|x'\|y)}{\|x'\|},
\]
so $\xi_{x''}$ attains its infimum at $0$. Since
$\xi_{x''}(0)=\|x''\|=1$, we get $\|x''-y\|=\xi_{x''}(y)\ge 1$ for
every $y\in E_n$. Therefore, the set $E_{n+1}=E_n\cup\{x''\}$ is
$1$-separated. Hence, $\bigcup_{n=1}^\infty E_n$ is an infinite
$1$-separated subset of the unit sphere in $X$.
\end{proof}

\begin{prop}\label{lt1}
Let $X$ be a normed vector space over $\RR$ of infinite dimension. Then its
coarse entropy is infinite.
\end{prop}

\begin{proof}
Fix $\delta, R>0$, and assume that $n\in \NN$ is sufficiently large that $n\delta > R$. By \cref{ll1}, the $R$-ball in $X$ contains an infinite $R$-separated subset $A$. For every $a\in A$, the sequence $p_a\defeq (0, \frac{a}{n}, \frac{2a}{n}, \ldots, a)$ is a $\delta$-pseudoorbit of $\id_X$, and the family $\{p_a \st a \in A\}$ is clearly $R$-separated.
Therefore the coarse entropy of $X$ is infinite.
\end{proof}

Let us make three simple yet fundamental observations.
 
\begin{lem}\label{subspace}
If $X\subseteq Y$ are metric spaces, then $h_\infty(X)\leq h_\infty(Y)$.
\end{lem}
\begin{proof}
Let $x_0\in X$, and let $n\in \NN$, $R>0$, and $\delta>0$. Clearly, $P(\id_X, n, R, \delta, x_0) \subseteq P(\id_Y, n, R, \delta, x_0)$, so if $O$ is an $R$-separated subset of $P(\id_X, n, R, \delta, x_0)$, then it is an $R$-separated subset of $P(\id_Y, n, R, \delta, x_0)$. Hence, $s(\id_X, n, R, \delta, x_0) \leq s(\id_Y, n, R, \delta, x_0)$, and thus $h_\infty(X)\leq h_\infty(Y)$.
\end{proof}

\begin{lem}\label{coequiv}
If two spaces are coarsely equivalent, then they have the same
coarse entropy.
\end{lem}

\begin{proof}
Clearly, if two spaces are coarsely equivalent (we recall the definition of coarse equivalence in \cref{ce}), then their identities
are coarsely conjugate in the sense of~\cite{GM}, so by Corollary~3.2
of~\cite{GM}, they have the same coarse entropy.
\end{proof}

\begin{lem}\label{zeroinf}
The coarse entropy of any metric space is either zero or infinity.
\end{lem}

\begin{proof}
Clearly, the identity is a controlled map (in the sense of~\cite{GM}),
so by Theorem~3.6 of~\cite{GM}, the coarse entropy of its $k$-th
iterate is equal to $k$ times the coarse entropy of the identity.
However, any iterate of the identity is the identity itself, so its
coarse entropy must be zero or infinity.
\end{proof}

\Cref{isomEmbedding} and \cref{zeroinf} show in particular that the coarse entropy of any isometric self-embedding is either zero or infinity.

In the proofs of \cref{lt1} and \cref{subspace,zeroinf,coequiv}, we used the fact from \cref{isomEmbedding} that $h_\infty(X) = h_\infty(\id_X)$. Since we will often resort to the special case of the identity maps in the sequel, let us simplify terminology and notation. We
will call $\delta$-pseudoorbits of the identity \emph{$\delta$-paths}
and for $f=\id_X$ skip $f$ in the notation for $s(\cdot)$ and $P(\cdot)$. 
After these simplifications, formula \eqref{firstInThePaper} boils down to the following:
\begin{gather}\label{spaceEntropySeparated}
\begin{aligned}
h_\infty(X)&=\lim_{\delta\to \infty} \lim_{R \to \infty} \limsup_{n \to \infty} \fdfrac{1}{n} \log s(n,R,\delta,x_0), \text{ where}\\
s(n,R,\delta, x_0) &= \sup \{ |O| : O \subseteq P(n,\delta, x_0) \text{ is an }R\text{-separated subset}\}.
\end{aligned}
\end{gather}

Let $A\geq 0$, $X$ be a metric space, and $Z\subseteq X$ be a subset.
The subset~$Z$ is \emph{$A$-dense} (also known as \emph{$A$-spanning}) in~$X$ if for every $x\in X$ there is $z\in Z$ with $d(x,z) < A$.  When the constant~$A$ is not important, we will also say that~$Z$ is \emph{coarsely dense} in~$X$.
It was proved in \cite{GM} that in the definition of the coarse entropy of a map, $R$-separated sets can be replaced with $R$-dense sets. Namely, instead of the supremum $s(f,n,R,\delta,x_0)$ of the cardinalities of
$R$-separated sets of $\delta$-pseudo\-orbits of~$f$ of length~$n$
starting at~$x_0$, in formula \eqref{firstInThePaper}  one can take the minimum $r(f,n,R,\delta,x_0)$ of the cardinalities of
$R$-dense sets of $\delta$-pseudoorbits of~$f$ of length~$n$
starting at~$x_0$. In the special case of coarse entropy of spaces, we obtain
\begin{gather}\label{spaceEntropyDense}
\begin{aligned}
h_\infty(X)&=\lim_{\delta\to \infty} \lim_{R \to \infty} \limsup_{n \to \infty} \fdfrac{1}{n} \log r(n,R,\delta,x_0), \text{ where}\\
r(n,R,\delta, x_0) &= \min \{ |O| : O \subseteq P(n,\delta, x_0) \text{ is an }R\text{-dense subset}\}.
\end{aligned}
\end{gather}

\section{Metric spaces}\label{sec:metric}

We hope that the paper will be read by some people for whom some definitions in this section will be new, although an expert in coarse geometry can only skim it in order to learn our notation and conventions.

Throughout, the symbol $\NN$ will denote the set of positive integers $\{1,2,3,\ldots\}$.

\begin{defi}\label{defi:graph} A \emph{graph} is a pair $(V,E)$, where $V$ is any set, called the set of vertices, and $E$ is a set of two-element subsets (called \emph{edges}) $\{v,w\}$ of $V$.

A sequence $(v_0, \ldots, v_n)$ of elements of $V$ is called a \emph{path} if $\{v_{i-1}, v_i\}$ is an edge for every $i\in \{1,\ldots, n\}$. The integer $n$ is then called the \emph{length} of the path. We will say that a path $(v_0, \ldots, v_n)$ \emph{connects} $v_0$ and $v_n$. If for every pair of vertices $v\neq w \in V$ there is a path connecting them, we say that the graph $(V,E)$ is \emph{connected}.

When $(V,E)$ is a connected graph, the distance $d(v,w)$ between $v,w\in V$ is the minimal length of a path connecting them. This is a metric called the \emph{path metric}. 
\end{defi}

Note that the set of edges can be reconstructed from the metric by the identity $E = \{ \{v, w\} \st d(v, w) = 1\}$. We will often neglect the distinction between the graph $(V,E)$ and the corresponding metric space $(V, d)$.

Since we are interested in metric spaces, all graphs that we will consider will be connected, although sometimes we will need to prove it. Usually, the \emph{degree} of every vertex $v$ (that is, the number of edges containing this vertex) will be finite, and often in fact bounded uniformly in~$v$.

\begin{defi} A \emph{weighted graph} is a graph $(V,E)$ together with a function $\omega\colon E\to (0, \infty)$, called the \emph{weight function}.

The \emph{weight of a path} $(v_0, \ldots, v_n)$ in a weighted graph $(V,E,\omega)$ is the sum $\sum_{i=1}^n \omega(\{v_{i-1}, v_i\})$.

When $(V,E)$ is connected, \emph{the weighted distance} $d(v,w)$ between $v,w\in V$ is the infimum of the weights of paths connecting them.
\end{defi}

\begin{longRem} The function $d\colon V\times V\to [0,\infty)$ given by weighted distances is a metric in many cases of interest (and a pseudometric in general), for example if one assumes that $\omega(e) \geq 1$ for every $e\in E$ or that the degree of every vertex is finite. When we write about \emph{the metric} on a weighted graph without further comments, we mean this metric (rather than the path metric of the underlying graph).
\end{longRem}

Let $(X,d)$ be a metric space and $A,B \subseteq X$ be nonempty. By slightly abusing notation, we will denote
\[d(A,B) \defeq \inf_{a\in A,\, b\in B} d(a,b).\]
For $x\in X$ and $B\subseteq X$, the notations $d(x,B)$ and $d(B, x)$ are shorthands for $d(\{x\}, B)$.

We will often briefly write that $X$ is a metric space, and then it should be inferred that the metric on $X$ is denoted by $d$. In particular, we will often use the same letter $d$ to denote metrics on different metric spaces. In cases when this could lead to confusion, we will temporarily add the subscript and write $d_X$.

Throughout, a ball $B(x,r)$ centred at $x\in X$ of radius $r > 0$ in a metric space $X$ will mean the \emph{closed} ball, that is, $B(x,r) = \{x' \in X \st d(x,x')\leq r \}$. Occasionally, when the space $X$ is not obvious from the context, we will include it in the notation by writing $B_X(x,r)$. In \cref{stepBallDefinition}, for a number $\delta > 0$ and $r\in \NN$, we will define a set that will similarly be denoted by $B_\delta(x,r)$, but this should not lead to confusion.

\begin{defi}
Given two maps $f,g\colon S\to X$ from any set $S$ to a metric space $(X,d)$, we will say that $f$ and $g$ are \emph{close} if their supremum distance is finite, namely $\sup_{s\in S} d(f(s), g(s)) < \infty$.
\end{defi}

\begin{defi}
Let $C\geq 1$ and $A\geq 0$. A map $f\colon X\to Y$ between metric spaces is called a \emph{$(C,A)$-quasi-isometric embedding} if the following inequality holds for every $x,x'\in X$:
\[ C^{-1} d(x,x') - A \leq d(f(x), f(x')) \leq C d(x,x') + A.\]
Additionally, if for every $y\in Y$, there is $x\in X$ with $d(y, f(x))\leq A$, then $f$ is called a \emph{$(C,A)$-quasi-isometry}\footnote{Note that this additional condition follows if $f(X)$ is $A$-dense in $Y$. We use this condition instead of referring to coarse density because we prefer the strict inequality in the definition of coarse density, and the nonstrict inequality in the definition of quasi-isometry, so that an isometry is a $(1,0)$-quasi-isometry rather that a $(1,\eps)$-quasi-isometry.}.
\end{defi}

When the constants $C$ and $A$ are not important, we will refer more briefly to a quasi-isometric embedding or a quasi-isometry, respectively. If there exists a quasi-isometry $f\colon X\to Y$, we will say that $X$ and $Y$ are quasi-isometric. It is easy to see that the composition of two quasi-isometries is a quasi-isometry (and accordingly for quasi-isometric embeddings), and it is a standard exercise to show that when $f\colon X\to Y$ is a quasi-isometry, then there exists a
quasi-isometry $g\colon Y\to X$.
One can moreover require that $f\circ g$ is close to $\id_Y$ and $g\circ f$ is close to $\id_X$.

\begin{longEx}\label{word-metric} Let $\Gamma$ be a group and let $S$ be a generating set for~$\Gamma$. The vertex set of the Cayley graph $\cay(\Gamma, S)$ is $\Gamma$, and two vertices $\gamma, \gamma'$ form an edge if $\gamma = \gamma' s$ or $\gamma' = \gamma s$ for some $s\in S$. Because $S$ is a generating set, $\cay(\Gamma, S)$ is a connected graph, so we can consider the path metric $d$ on its set of vertices as in \cref{defi:graph}.
This is known as the \emph{(left-invariant) word metric} on $\Gamma$ associated with $S$.

If $S$ is finite and $S'$ is another finite generating set for $\Gamma$, it induces a possibly different word metric~$d'$, but the identity map $(\Gamma, d) \to (\Gamma, d')$ is always a quasi-isometry (in fact, it is a bi-Lipschitz map). Hence, finitely generated groups have metrics that are canonical up to quasi-isometry.
\end{longEx}

We have the following more general definition.

\begin{defi}\label{ce} A map $f\colon X\to Y$ between metric spaces is called a \emph{coarse embedding} if there exist two nondecreasing functions $\rho_-,\rho_+\colon [0,\infty)\to \RR$ such that $\lim_{r\to\infty} \rho_-(r)=\infty$ and the following inequality holds for every $x,x'\in X$:
\[ \rho_-(d(x,x')) \leq d(f(x), f(x')) \leq \rho_+(d(x,x')).\]
Additionally, if the image $f(X)$ is coarsely dense in $Y$, then $f$ is called a \emph{coarse equivalence}.
\end{defi}

In particular, if $f$ is a $(C,A)$-quasi-isometric embedding, then we can put $\rho_-(r) = C^{-1}r - A$ and $\rho_+(r)=Cr+A$ to see that indeed $f$ is a coarse embedding. Similarly as in the special case of quasi-isometries, the composition of coarse embeddings (respectively: coarse equivalences) is a coarse embedding (respectively: a coarse equivalence), and every coarse equivalence $f\colon X\to Y$ admits a map $g\colon Y\to X$ such that $f\circ g$ is close to $\id_Y$ and $g\circ f$ is close to $\id_X$. Such a map $g$ is uniquely defined up to closeness of maps, and it is necessarily a coarse equivalence. Hence, if there is a coarse equivalence $f\colon X\to Y$, it makes sense to call $X$ and $Y$ \emph{coarsely equivalent}.

\begin{longRem}\label{coarselyMonotone}
By definition, if $f\colon X\to Y$ is a coarse embedding, then for $Z=f(X)$ the map $f\colon X\to Z$ is a coarse equivalence. Hence, by combining \cref{subspace,coequiv} we get that coarse entropy is monotone under coarse embeddings.
\end{longRem}

\begin{defi} A metric space $X$ is called \emph{locally finite} if the cardinality of every ball in $X$ is finite.
\end{defi}

\begin{defi}\label{boundedGeometry} A metric space $X$ \emph{has bounded geometry} if for every positive $r$ there is a constant $C(r) \in \NN$ such that the cardinality of the ball $B(x,r)$ is bounded by $C(r)$ for all $x\in X$.
\end{defi}

Note that having bounded geometry in the sense of \cref{boundedGeometry} is not invariant under coarse equivalences or even quasi-isometries: for example, the inclusion $\ZZ \subseteq \RR$ is a quasi-isometry, yet $\ZZ$ has bounded geometry (one can take $C(r)=2r+1$) and $\RR$ does not (every ball is uncountable). Moreover, this property is not invariant under quasi-isometries even if we restrict to graphs: consider for example two trees, $T_1$~consisting only of an infinite trunk (i.e.\ $\NN$ with the usual graph structure), and $T_2$ consisting of the trunk together with $2^k$ leaves (i.e.\ degree-one vertices) attached at $k$, for all integers
$k\in \NN$. Then $T_1$ and $T_2$ are coarsely equivalent, but the cardinalities of balls
of radius $r$ are bounded by $2r+1$ in $T_1$ and unbounded in $T_2$.

However, it is possible (see \cref{ceBoundedGeometry}) to internally characterise the property of being coarsely equivalent to a metric space with bounded geometry. Nonetheless, \cref{boundedGeometry} is commonly used in the literature because it may be more convenient to replace the considered space by a space satisfying \cref{boundedGeometry} than to deal with the original space throughout a proof.

\begin{prop}\label{ceBoundedGeometry} The following conditions are equivalent for a metric space~$X$:
\begin{enumerate}[(i)]
\item\label{ceBoundedGeometry-one} there exists $s > 0$ such that for every $D > 0$ there exists a constant $C'(D)\in \NN$ such that the cardinality of any $s$-separated subset of $X$ of diameter at most $D$ is bounded by $C'(D)$,
\item\label{ceBoundedGeometry-two} $X$ contains a coarsely dense subset with bounded geometry,
\item\label{ceBoundedGeometry-three} $X$ is quasi-isometric to a space with bounded geometry,
\item\label{ceBoundedGeometry-four} $X$ is coarsely equivalent to a space with bounded geometry.
\end{enumerate}
\end{prop}
\begin{proof} $\ref{ceBoundedGeometry-one}\Rightarrow \ref{ceBoundedGeometry-two}$. Let $Z$ be a maximal $s$-separated subset of $X$ (such sets exist by the Kuratowski--Zorn lemma). By maximality, $Z$ is $s$-dense in $X$. We put $C(r)=C'(2r)$ for all positive $r$. Now, for any $z\in Z$ and any $r$ the ball $B_Z(z, r)$ has diameter at most $2r$ and is $s$-separated as a subset of $Z$, and hence its cardinality is at most $C'(2r)$ by the assumption: this shows that $Z$ has bounded geometry.

The implications $\ref{ceBoundedGeometry-two}\Rightarrow \ref{ceBoundedGeometry-three}$ and $\ref{ceBoundedGeometry-three}\Rightarrow \ref{ceBoundedGeometry-four}$ are immediate.

$\ref{ceBoundedGeometry-four}\Rightarrow \ref{ceBoundedGeometry-one}$. By the assumption, there is a space $Z$ with bounded geometry and a coarse equivalence $f\colon X\to Z$. Let $C(r)$ be the constants from \cref{boundedGeometry} for~$Z$ and let $\rho_\pm$ be the functions from \cref{ce} for~$f$. Since $\lim_{t\to \infty} \rho_-(t) = \infty$, we can pick $s > 0$ such that $\rho_-(s) > 0$. Define $C'(D) = C(\rho_+(D))$ for every positive~$D$. Now, fix a positive~$D$ and let $Y\subseteq X$ be any $s$-separated subset of diameter at most~$D$. Since $\rho_-(s) > 0$, $f$ is injective on $Y$, and hence the cardinality of $Y$ equals the cardinality of $f(Y)$. However, the diameter of $f(Y)$ is bounded by $\rho_+(D)$, and hence $f(Y)$ is contained in the ball $B(z, \rho_+(D))$ for any $z\in f(Y)$. By the assumption, the cardinality of such balls is at most $C(\rho_+(D))$.
\end{proof}

\begin{defi}
For a metric space $X$ and a basepoint $x\in X$, the \emph{volume growth} is the function
\[\NN \owns r \mapsto |B(x,r)| \in \NN\cup\{\infty\},\]
where $|A|$ denotes the cardinality of a set $A$.
\end{defi}
Clearly, the above function may depend on~$x$ and is sensitive to a change of~$d$, so it is customary to consider 
the equivalence class of this function, where two non-decreasing functions $f,g\colon \NN\to \NN\cup\{\infty\}$ are identified if there exists a constant $D\in \NN$ such that
\[f(r)\leq Dg(Dr) \quad\text{and}\quad g(r)\leq Df(Dr)\]
for all $r\in \NN$. In particular, one says that $X$ has \emph{exponential growth} if its volume growth is equivalent to $r\mapsto \exp(r)$.

\begin{defi}\label{delta-paths}
Recall that a finite sequence $(x_0, \ldots, x_n)$ in a metric space $X$ is called a $\delta$-path for $\delta > 0$ if $d(x_{i-1}, x_i)\leq \delta$ for every $i\in \{1,\ldots, n\}$. We say that such a $\delta$-path \emph{connects} $x_0$ and $x_n$, and we call the number~$n$ the \emph{length} of the $\delta$-path.
\end{defi}

\begin{defi} Let $X$ be a metric space. The binary relation on $X$  consisting of those pairs $(x,x')\in X^2$ for which there is a $\delta$-path connecting $x$ and $x'$ is an equivalence relation, and the equivalence classes of this relation are called \emph{$\delta$-components} of~$X$. For $x\in X$, the $\delta$-component containing $x$ will be denoted by $[x]_\delta$.
\end{defi}

Recall that a metric space $X$ is \emph{geodesic} if for every two points $x,x'\in X$ there is an isometric embedding $p\colon [0,d(x,x')]\to X$ of the interval $[0,d(x,x')]$ such that $p(0)=x$ and $p(d(x,x'))=x'$. Quasigeodesic spaces are defined similarly as follows.

\begin{defi}\label{defi:qg}
A metric space $X$ is \emph{quasigeodesic} if there exist constants $C\geq 1$ and $A\geq 0$ such that for every two points $x,x'\in X$ there exists a $(C,A)$-quasi-isometric embedding $p\colon [0,d(x,x')]\to X$ with $p(0)=x$ and $p(d(x,x'))=x'$. We will briefly say that $p$ \emph{connects} $x$ and $x'$. The map $p$ will also be called a \emph{$(C,A)$-quasigeodesic}.
\end{defi}

In many applications, one is interested only in quasigeodesic spaces. In particular, every connected graph is quasigeodesic, and hence every finitely generated group is quasigeodesic when equipped with the word metric (see \cref{word-metric}).

\begin{prop}\label{qg} The following conditions are equivalent for a metric space $(X,d_X)$:
\begin{enumerate}[(i)]
\item\label{qg1} $X$ is quasigeodesic,
\item\label{qg2} there exists $\delta > 0$ such that for all points $x,x'\in X$ there is a $\delta$-path connecting them of length at most $\ceil{d_X(x,x')}$,
\item\label{qg3} $X$ is quasi-isometric to a connected graph.
\end{enumerate}
Moreover, if $X$ is locally finite, then in \ref{qg3} we can assume the degrees of vertices to be finite. Similarly, if $X$ is coarsely equivalent to a space of bounded geometry, then we can assume the degrees of vertices to be uniformly bounded.
\end{prop}

The equivalence of \ref{qg1} and \ref{qg3} is well-known among experts, and the characterisation from \ref{qg2} appears in the third author's thesis \cite{thesis}. We provide a proof for self-containment.

\begin{proof}[Proof of \cref{qg}]
$\ref{qg1}\Rightarrow \ref{qg2}$. Let $C,A$ be the constants from \cref{defi:qg} and define $\delta = C+A$. Let $x,x'\in X$ be distinct and let $p\colon [0,d_X(x,x')]\to X$ be a $(C,A)$-quasigeodesic connecting $x$ and $x'$. Let $c=d_X(x,x')/\ceil{d_X(x,x')} \leq 1$. Then the sequence
\[\big(p(0), p(c), \ldots, p(\ceil{d_X(x,x')} \cdot c)\big)\]
is a $\delta$-path connecting $p(0)=x$ and $p(\ceil{d_X(x,x')} \cdot c)=x'$.

\smallskip $\ref{qg2}\Rightarrow \ref{qg3}$. Let a graph $(V,E)$ be defined as follows: $V=X$ and
\[E = \{ \{x,x'\} \st x, x'\in X \text{ and } 0 < d_X(x,x') \leq \delta \}.\]
It follows from the assumption and the definition of $E$ that the graph $(V,E)$ is connected, so let us consider the path metric $d_{(V,E)}$ on $V$. We will show that the identity map $f\colon (X,d_X)\to (V,d_{(V,E)})$ is a quasi-isometry. Since $f$ is surjective, it is obvious that $f(X)$ is coarsely dense in $V$. For any $x,x'\in X$, there is a $\delta$-path $p$ in $X$ of length at most $\ceil{d_X(x,x')}$. However, by the definition of $E$, the sequence $p$ is a path in $(V,E)$, and hence
\begin{equation}\label{qg23-1}
d_{(V,E)}(f(x), f(x')) \leq \ceil{d_X(x,x')} \leq d_X(x,x') + 1.
\end{equation}
In the other direction, by the definition of path metric, for any $x,x'\in X$ there is a path $q = (q_0, \ldots, q_n)$ in $(V,E)$ of length $d_{(V,E)}(f(x), f(x')) \eqdef n$ connecting $q_0=f(x) = x$ and $q_n=f(x') = x'$. Hence we get
\begin{equation}\label{qg23-2}
\fdfrac{1}{\delta}d_X(x,x') \leq \fdfrac{1}{\delta}\textstyle\sum_{i=1}^n d_X(q_{i-1}, q_i) \leq n = d_{(V,E)}(f(x), f(x')),
\end{equation}
so by combining inequalities \labelcref{qg23-1,qg23-2} we conclude that $f$ is a $(\max(1,\delta),1)$-quasi-isometry.

\smallskip $\ref{qg3}\Rightarrow \ref{qg1}$. By the assumption, there exist a graph $(V,E)$, constants $C$ and $A$, and a $(C,A)$-quasi-isometry $f\colon V\to X$. Let $x, x' \in X$ be distinct. Since $f(V)$ is $A$-dense in $X$, there exist vertices $v, v'\in V$ such that $d_X(x, f(v)) \leq A$ and $d_X(x', f(v')) \leq A$.

Note that the map $p_{x,x'}\colon [0,d_X(x,x')] \to X$ given by $p_{x,x'}(r) = x$ for $r<d(x,x')$ and $p_{x,x'}(d(x,x')) = x'$ is a $(1, d(x,x'))$-quasigeodesic. Hence, we will further assume that  $d(x,x')\geq 1$. Similarly, if $v=v'$, then $d(x,x')\leq 2A$, and the map $p_{x,x'}$ is a $(1, 2A)$-quasigeodesic, so we will also assume that $v$ and $v'$ are distinct, that is, $d(v,v')\geq 1$.

Let $(v_0, \ldots, v_{d_{(V,E)}(v,v')})$ be a path of minimal length connecting $v_0=v$ and $v_{d_{(V,E)}(v,v')} = v'$. Note that the map $p\colon [0,d_{(V,E)}(v,v')]\to V$ given by $p(r) = v_{\floor r}$ is a $(1,1)$-quasigeodesic. Consequently, the composition $f\circ p\colon [0,d_{(V,E)}(v,v')]\to X$ is a $(C, C+A)$-quasi-isometric embedding.

Since $d_X(x, f(v)) \leq A$ and $d_X(x', f(v')) \leq A$, we conclude that
\[\frac{d_X(x,x')}{d_{(V,E)}(v,v')} \leq \frac{d_X(f(v), f(v')) + 2A}{d_{(V,E)}(v,v')} \leq \frac{C d_{(V,E)}(v,v')  + 3A}{d_{(V,E)}(v,v') } \leq C + 3A,\]
and similarly
\begin{align*}
\frac{d_{(V,E)}(v,v')}{d_X(x,x')} &\leq \frac{C d_X(f(v),f(v')) + CA}{d_X(x,x')} \\
&\leq \frac{C (d_X(x,x') + 2A) + CA}{d_X(x,x')} \leq 3CA + C.
\end{align*}
In particular, the homothety $h$ mapping $[0,d_X(x,x')]$ onto $[0,d_{(V,E)}(v,v')]$ is a $(D,0)$-quasi-isometry for $D = \max(C+3A, 3CA+C)$. Consequently, the composition $f\circ p\circ h\colon [0, d_X(x,x')]\to X$ is a $(CD, C+A)$-quasi-isometric embedding. Define $q\colon [0, d_X(x,x')]\to X$ by
\[ q(r) = \begin{cases}
x & \text{ if } r=0,\\
f\circ p\circ h(r) & \text{ if } r\in (0, d_X(x,x')),\\
x' & \text{ if } r=d_X(x,x').
\end{cases}\]
Then, $q$ is a desired $(CD, C+3A)$-quasigeodesic connecting $x$ and $x'$.

\smallskip For the ``moreover'' part, note that if $X$ is locally finite, then the graph $(V,E)$ constructed in the proof of the implication $\ref{qg2} \Rightarrow \ref{qg3}$ has finite degrees of vertices. Indeed, this is because the set of neighbours of any vertex $v\in V=X$ is contained in the ball of radius $\delta$ around $v\in X$.

Similarly, if $X$ has bounded geometry, then $(V,E)$ has a uniform bound on the degree of vertices. More generally, by the implication $\ref{ceBoundedGeometry-four}\Rightarrow \ref{ceBoundedGeometry-three}$ in \cref{ceBoundedGeometry}, when $X$ is coarsely equivalent to a space with bounded geometry, there is a space $Y$ with bounded geometry that is quasi-isometric to $X$. It follows from the equivalence of \ref{qg1} and \ref{qg3} that a space quasi-isometric to a quasigeodesic space is itself quasigeodesic, and hence $Y$ is quasi-geodesic. Applying the first sentence of this paragraph to $Y$ gives a graph $(V,E)$ quasi-isometric to $X$ with a uniform bound on the degrees of vertices.
\end{proof}

In the realm of quasigeodesic spaces, every coarse equivalence is a quasi-isometry (see e.g.\ Corollary~1.4.14 in \cite{NY}), although many naturally occurring coarse embeddings are not quasi-isometric. For example, if groups $\Gamma, \Gamma'$ are generated by finite sets $S,S'$ respectively, and $\Gamma'$ is isomorphic to a subgroup $\Lambda \leq \Gamma$, then $\Lambda$ is called \emph{distorted} if the corresponding embedding $\Gamma' \simeq \Lambda \leq \Gamma$ is not quasi-isometric (see for instance \cite{LMR}). Given a finitely generated group~$\Gamma$, one can consider the set of its coarse embeddings into the Hilbert space~$\ell_2$ and quantify how much each of them fails to be a quasi-isometry: the infimal degree of this failure is an invariant of~$\Gamma$ known as the \emph{Hilbert-space compression} \cite{Guentner--Kaminker}.

In fact, even without embedding it into any other space, it is very easy to obtain a space which is not quasigeodesic: for example, the real line $\RR$ with the metric $d'$ given by $d'(x,x') = \sqrt{|x-x'|}$ is not quasigeodesic. Nonetheless, it satisfies the equivalent conditions of the following \cref{ceqg}.

\begin{prop}\label{ceqg} The following conditions are equivalent for a metric space $(X,d_X)$:
\begin{enumerate}[(i)]
\item\label{ceqg1} $X$ is coarsely equivalent to a quasigeodesic space,
\item\label{ceqg2} there exists a nondecreasing function $\rho\colon [0,\infty)\to \NN$ and a constant $\delta > 0$ such that for every $x,x'\in X$ there is a $\delta$-path connecting them of length at most $\rho(d_X(x,x'))$,
\item\label{ceqg3} $X$ is coarsely equivalent to a connected graph.
\end{enumerate}
\end{prop}

\Cref{ceqg} is mentioned for comparison with \cref{qg}, but it will not be used later, so we omit the proof. The details can be found in \cite{thesis}, where the characterisation from item~\ref{ceqg2} was introduced and used in the construction of the first counterexamples to the coarse Baum--Connes conjecture not coarsely equivalent to a family of graphs.

\section{Examples}\label{sec:examples}

Before we pass to our main results, we want to present some examples that will show what properties we cannot expect. More specifically, \cref{s:ze,s:is} contain examples of spaces of bounded geometry such that
\begin{itemize}[—]
\item the volume growth is exponential, and the coarse entropy is zero (\cref{boundedComponents,first-example,strangeMetricsExample}),
\item the volume growth is linear, and the coarse entropy is infinite (\cref{second-example}).
\end{itemize}

In the first case, the reason why exponential growth is insufficient to imply infinite coarse entropy is that the spaces considered are far from being geodesic. In the second case, the example is a graph and demonstrates that for spaces that are not sufficiently homogeneous (e.g.\ graphs that are not vertex-transitive\footnote{A graph $(V,E)$ is \emph{vertex-transitive} if for every $v,v'\in V$, there is an automorphism $f$ of $(V,E)$ such that $f(v) = v'$.}), the usual notion of volume growth is not adequate for coarse entropy.

Furthermore, \cref{s:zu} provides an example of
\begin{itemize}[—]
\item a space that does not have bounded geometry, and its coarse entropy is zero (\cref{unboundedZero}).
\end{itemize}
The reason is similar to that in the first case above. Quite interestingly, the example is $1$-connected, but it is not quasi-geodesic.

\subsection{Zero coarse entropy of exponentially growing spaces}\label{s:ze}

\begin{prop}\label{boundedComponents}
Let $X$ be a metric space and assume that for every $\delta > 0$ every $\delta$-component of $X$ is bounded. Then, the coarse entropy of $X$ is zero. 
\end{prop}
\begin{proof}
Fix $x\in X$ and $\delta > 0$, and let $D$ be the diameter of the $\delta$-component $[x]_\delta$ of~$x$. By the definition of $\delta$-components, any $\delta$-path starting at $x$ cannot leave $[x]_\delta$, and hence for all $R>D$ any $R$-separated set of $\delta$-paths starting at $x$ contains at most 1 point. Thus, the coarse entropy of $X$ vanishes.
\end{proof}

A metric $d$ on a metric space $X$ is called an \emph{ultrametric} if it satisfies the following strong form of the triangle inequality:
\[\forall{x,y,z\in X}\; d(x,z) \leq \max(d(x,y), d(y,z)).\]
A metric space whose metric is an ultrametric is called an \emph{ultrametric space}.

\begin{cor}\label{ultrametric} The coarse entropy of ultrametric spaces is zero.
\end{cor}
\begin{proof}
If $X$ is an ultrametric space, then the diameter of any $\delta$-component is bounded by $\delta$.
\end{proof}

\begin{longEx}\label{first-example} Consider the space $X$ of sequences in $\prod_{k\in \NN} \{0, k\}$ with finitely many non-zero entries, and equip $X$ with the supremum metric:
\[d((x_k), (y_k)) = \max_k |x_k - y_k|.\]
\begin{claim}
The metric $d$ on $X$ from \cref{first-example}
is an ultrametric, and hence \cref{ultrametric} applies giving $h_\infty(X) = 0$. However, the volume growth of $X$ is exponential, and $X$ has bounded geometry.
\end{claim}
\end{longEx}
\begin{proof}
To show that $d$ is an ultrametric, it suffices to notice that when the last coordinate such that $(x_k)$ differs from $(y_k)$ is $m$, and the last coordinate such that $(y_k)$ differs from $(z_k)$ is $m'$, then $(x_k)$ and $(z_k)$ agree on coordinates larger than $\max(m,m')$.

We will now show that the volume growth of $X$ is exponential, and $X$ has bounded geometry. Let $x\in X$. The map $\iota_x\colon X\to X$ given by $(\iota_x(y))_k=|x_k-y_k|$ maps $x$ to~$0$ and is an isometry because for $y,z\in X$ one has
\[|(\iota_x(y))_k-(\iota_x(z))_k|=||x_k-y_k|-|x_k-z_k||=|y_k-z_k|.\]
This means that $X$ is homogeneous: any point can be mapped by an isometry to 0. Consequently, in order to show that it has bounded geometry it suffices to show that it is locally finite.

Now observe that the $n$-ball around $x = 0$ is the product $\prod_{k\leq n} \{0,k\} \times \prod_{k>n}\{0\}$, which contains $2^n$ elements. This shows that $X$ is locally finite (and hence, by the above, has bounded geometry), and that its volume growth is exponential.
\end{proof}

Apart from ultrametric spaces, the following two important classes of metric spaces satisfy the assumption of \cref{boundedComponents}, i.e.\ that their $\delta$-components are bounded for every $\delta > 0$:
\begin{itemize}[—]
\item\label{asdimzero} metric spaces of asymptotic dimension zero (this is equivalent to the fact that for every fixed $\delta$ the diameters of all $\delta$-components are \emph{uniformly} bounded),
\item\label{coarseunion} spaces $\bigsqcup_n X_n$ obtained as \emph{coarse disjoint unions} of a sequence of bounded spaces $X_n$.
\end{itemize}
Recall that a \emph{coarse disjoint union} of bounded metric spaces $(X_n, d_n)$ is the set-theoretic disjoint union $\bigsqcup_n X_n$  equipped with any metric $d$ (all such metrics will be coarsely equivalent) such that
\begin{equation}\label{coarseDisjointUnion}
\forall n\; d_{|X_n\times X_n} = d_n \text{~~and~~} 
\lim_{n\to \infty} d(X_n, \textstyle\bigcup_{m\neq n} X_m) = \infty.
\end{equation}
One can check that the metric $d$ given by $d(x, y) = d_n(x, y)$ for $x,y\in X_n$ and $d(x,y) = \max(n,m, \diam X_n, \diam X_m)$ for $x\in X_n \neq X_m \owns y$ always satisfies \eqref{coarseDisjointUnion}.

In particular, $\delta$-components are bounded for the following important examples:
\begin{itemize}[—]
\item Countable locally finite groups (that is, countable groups whose finitely generated subgroups are finite) equipped with any proper\footnote{\emph{Properness} of a discrete metric is the same as \emph{local finiteness} (finiteness of the cardinality of any ball). We used the word \emph{proper} to avoid possible confusion resulting from talking about a ``\emph{locally finite} metric on a \emph{locally finite} group''.} translation-invariant metric have asymptotic dimension zero~\cite{Smith}. Note that all such metrics are coarsely equivalent, and the standard way to obtain one is to fix an infinite generating set $S$ and take the ``weighted'' word-length metric with the lengths of generators increasing to infinity.
\item Box spaces, i.e.\ coarse disjoint unions $\bigsqcup_n X_n$, where $X_n = \Gamma / \Gamma_n$ are finite quotients of a finitely generated group $\Gamma$.
\end{itemize}
In fact, in a somehow artificial way locally finite groups are also coarse disjoint unions as one can define $X_0 = \{e\}$ and $X_n = \langle s_1, \ldots, s_n\rangle \setminus \langle s_1, \ldots, s_{n-1}\rangle$, where $S=\{s_1, s_2, \ldots\}$.

\begin{longRem}
The notion of volume growth is a quasi-isometry invariant, while the metric on a countably-generated group is canonical only up to coarse equivalence (recall that quasi-isometry is a strictly finer notion of equivalence of metric spaces than coarse equivalence). For example, $X$ in \cref{first-example} can be equipped with the group structure of the locally finite group $\bigoplus_k \ZZ/2\ZZ$.
The metric from \cref{first-example} is proper and translation invariant, but there are other such metrics on $X$ yielding essentially different volume growth functions.

Indeed, on the one hand, $X$ is coarsely equivalent to a subspace of $\prod_{k\in \NN} \{0, \sqrt{k}\}$, for which the cardinality of the $n$-ball is $2^{n^2}$, so the volume growth is super-exponential. On the other hand, it is also coarsely equivalent to a subspace of $\prod_{k\in \NN} \{0, 2^{2^k}\}$ for which the cardinality of the $n$-ball for $n\geq 2$ is $2^{\floor{\log_2\log_2 n}} \leq \log_2(n)$, which is sublinear.
\end{longRem}

\Cref{boundedComponents} provides a sufficient condition for the coarse entropy to vanish even for exponentially growing spaces, but this condition is by no means a necessary one. In \cref{strangeMetricsExample} below, we provide an example of a space which is 1-connected (in particular, it does not satisfy the assumptions of \cref{boundedComponents}), and its coarse entropy vanishes despite the exponential growth of balls.

\begin{longEx}\label{strangeMetricsExample}
Recall that, as opposed to volume growth, the coarse entropy is a coarse invariant. Hence, if we consider $\ZZ$ with the metric given by $d(m, n) = \log_2(1+ |m-n|)$, then its coarse entropy is the same as for the usual metric, namely it vanishes (by \cite{GM}*{Theorem~4.3}). However, the cardinality of the $n$-ball equals $2^{n+1} - 1$, so the volume growth is exponential.
\end{longEx}

\subsection{Zero coarse entropy of spaces without bounded geometry}\label{s:zu}

\begin{defi}\label{stepBallDefinition} Let $(X,d)$ be a metric space and $\delta > 0$. For $x_0\in X$ and $n\in \NN$ denote
\[B_\delta(x_0, n) \defeq \{x_n \in X \st \exists\, \delta\text{-path } (x_0, x_1, \ldots, x_n) \}.\]
\end{defi}

The notation of \cref{stepBallDefinition} comes from the fact that $B_\delta(x_0, n)$ is the ball of radius $n$ around $x_0$ with respect to the metric (possibly assuming infinite values)
\[d_\delta(x, y) = \min \{ n \st \exists\, \delta\text{-path } (x, x_1 \ldots, x_{n-1}, y) \}\]
(the topology induced by the metric $d_\delta$ is always discrete; we do not claim that $(X,d)$ and $(X,d_\delta)$ are homeomorphic or coarsely equivalent).
We will sometimes refer to $B_\delta(x_0, n)$ as the \emph{$\delta$-ball of radius $n$ around $x_0$}.

\begin{longEx}\label{unboundedZero} Let $X$ be a weighted graph obtained as follows. Start with $\NN_0$ with the usual graph structure. Now, to every prime power $p^k \in \NN$ attach the cycle $G_p^k$ on $pk$ vertices, and, finally, between any two distinct vertices $v, w\in G_p^k$ add an edge with weight $p$.
\end{longEx}

\begin{claim}
The space $X$ from \cref{unboundedZero} is 1-connected and not coarsely equivalent to a space with bounded geometry, and its coarse entropy is zero.
\end{claim}
\begin{proof}
The fact that $X$ is 1-connected is obvious.

We will now prove that $X$ is not coarsely equivalent to a space with bounded geometry, that is, for every $s > 0$ there is $D > 0$ such that cardinalities of $s$-separated subsets of diameter at most $D$ are unbounded. Indeed, for a fixed $s > 0$ let $p$ be the smallest prime larger than $s$. Let $k\in \NN$ and consider the set $G_p^k$. By picking every $p$th vertex on the cycle $G_p^k$, one obtains an $s$-separated subset (in fact $p$-separated) of diameter $p$ and cardinality $k$. Letting $k$ go to infinity finishes the argument.

We will now prove that the coarse entropy vanishes. Fix $\delta > 1$ and $R > 8 \delta$. Define $r=R/4$ and $q=\floor{r/\delta} \geq 2$. Without loss of generality we can assume that $n$ is a multiple of $q$, namely that $n=mq$ for some $m\in \NN$. Pick a maximal $r$-separated subset $A$ of $X$ and any associated Voronoi partition $\{P_a : a\in A \}$ of $X$, i.e.\ a partition satisfying the condition that $x\in P_a$ implies that $d(x,a) = \min_{a'\in A} d(x,a')$. Consider the set $P(\delta, n, 0)$ (where $0\in \NN_0\subseteq X$) and define a map $e\colon P(\delta, n, 0) \to A^{m+1}$ by requiring that
\[e\left((x_i)_{i=0}^n\right) = (a_j)_{j=0}^m\text{, where $x_{jq}\in P_{a_j}$ for all $0\leq j\leq m$}.\]

First, note that $e$ is injective on $R$-separated subsets of $P(\delta, n, 0)$. Indeed, let us assume that $t, t' \in P(\delta, n, 0)$ and $e(t) = e(t')$, and we will show that the distance between $t$ and $t'$ is less than $R$. The assumption implies that for every $0\leq j\leq m$ we have $d(t_{jq}, t'_{jq})\leq 2r$. Since $t, t'$ are $\delta$-paths, it follows that for every $j\geq 0$ and every $0\leq i \leq q - 1$ such that $jq+i \leq n$ we have 
\[d(t_{jq+i}, t'_{jq+i}) \leq 2r + 2(q-1)\delta < 4r = R,\]
so indeed $e$ is injective on $R$-separated subsets.

Hence, in order to obtain an upper bound on the cardinality of $R$-separated subsets of $P(\delta, n, 0)$, it suffices to bound the cardinality of the image of~$e$. If $(x_i)$ is a $\delta$-path, $e((x_i)) = (a_j)_{j=0}^m$, and $0\leq j \leq m - 1$, then $a_{j+1}$ lies inside the $(\delta+r)$-ball $B_{\delta+r}(a_j, q)$ of radius $q$ (because $d(x_{jq}, a_j)\leq r$ and $d(x_{(j+1)q}, a_{j+1})\leq r$). For any $x\in X$, the $(\delta+r)$-ball $B_{\delta+r}(x, q)$ is contained in the union of
\begin{enumerate}[(a)]
\item\label{firstPart} an interval $I \subseteq \NN \subseteq X$ of at most $2q(\delta+r)+1$ points,
\item\label{secondPart} the sets $G_p^k$ for $p,k$ such that $p^k\in I$ and $p\leq \delta$, and
\item\label{thirdPart} intervals of at most $2q(\delta+r)+1$ points on every cycle $G_p^k$ such that $p^k \in I$ or $x\in G_p^k$.
\end{enumerate}
The cardinalities of the inverval $I$ in item~\labelcref{firstPart} and of intervals in item~\labelcref{thirdPart} are at most $Q\defeq 2q(\delta+r)+1$, and the number of intervals in item~\labelcref{thirdPart} is at most $Q$, so the cardinality of the union of all these intervals is at most $Q+Q^2$. However, the cardinalities of the sets $G_p^k$ occurring in item~\labelcref{secondPart} grow to infinity with $k$ (but we also know that there are at most $Q$ of them).

However, every set $G_p^k$ has diameter $p$, so the intersection of the $r$-separated set $A$ with every $G_p^k$ for $p\leq \delta$ (this is the crucial inequality in item~\labelcref{secondPart} above) consists of at most one point as $r\geq \delta \geq p$. Consequently, the intersection of any $(\delta+r)$-ball $B_{\delta+r}(x, q)$ with $A$ contains at most $2Q + Q^2$ points. This means that there are no more than $2Q + Q^2$ possible values of $a_{j+1}$ for a fixed $a_j$. Therefore, the cardinality of the image of $e$ is at most $(2Q+Q^2)^m$.

We conclude
\begin{align*}
h_\infty(X) &\leq \lim_{\delta\to \infty} \lim_{R \to \infty} \limsup_{m\to \infty}\frac{m \log (2Q+Q^2)}{mq}
\displaybreak[0]\\
&=
\lim_{\delta\to \infty} \lim_{R \to \infty} \frac{\log (2(2q(\delta+r)+1)+(2q(\delta+r)+1)^2)}{q}
\displaybreak[0]\\
&\leq
\lim_{\delta\to \infty} \lim_{R \to \infty} \frac{
\log (2(4R^2+1)+(4R^2+1)^2)
}{\floor{R/4\delta}} = 0
\end{align*}
where the last inequality uses the fact that $R$ is larger than each of $q$, $\delta$, and~$r$.
\end{proof}

\subsection{Infinite coarse entropy of slowly growing spaces}\label{s:is}

For a $\delta$-path $u = (z_0, \ldots z_n)$, let us denote by $u^{-1}$ the same $\delta$-path traversed backwards: $u^{-1} = (z_n, \ldots z_0)$. 

\begin{longEx}\label{second-example}
Let $(x_n)$ be a quickly growing sequence in $\NN$. Let $X$ consist of $\NN$ with the usual graph structure together with the full binary tree $T_n$ of height $n$ attached with the root at every $x_n$. It is clear that $X$ has bounded geometry. By an appropriate choice of $(x_n)$, we can require the volume growth to be linear, for example that  $\lim_{R\to\infty}\frac{ | B(0, R) | }{R} = 1$.
\end{longEx}
\begin{claim} Let $X$ be the space from \cref{second-example}. Then $h_\infty(X) = \infty$.
\end{claim}
\begin{proof}
Fix $\delta,R \in \NN$. Fix a $\delta$-path $o$ of length $\ceil{x_{2R} / \delta}$ connecting the origin to $x_{2R}$. Let $M_{2R}$ denote the set of vertices at level $R$ of the tree attached at $x_{2R}$. Clearly, the cardinality of $M_{2R}$ is $2^R$. For every $y \in M_{2R}$ pick a leaf $z\in T_n$ that is a descendant of~$y$ and choose a $\delta$-path $o_y$ of length $\ceil{2R / \delta}$ connecting $x_{2R}$ with $z$. Observe that the endpoints of $\delta$-paths $o_y, o_{y'}$ for $y\neq y' \in M_{2R}$ are $(2R+2)$-separated, in particular $R$-separated.

Let $p\in \NN$. By concatenating $o * (o_{y_1} * o_{y_1}^{-1}) * (o_{y_2} * o_{y_2}^{-1}) * \ldots * (o_{y_p} * o_{y_p}^{-1})$ for $y_i\in M_{2R}$ we obtain $(2^R)^p$ $R$-separated $\delta$-paths of length $\ceil{x_{2R} / \delta} + 2p\ceil{2R / \delta}$. Thus we get
\begin{align*}
h_\infty(X) &\geq \lim_{\delta\to \infty} \lim_{R\to\infty}  \lim_{p\to\infty}
\frac{\log(2^{Rp})}{\ceil{x_{2R} / \delta} + 2p\ceil{2R / \delta}} \\
&= \lim_{\delta\to\infty} \lim_{R\to\infty}  \lim_{p\to\infty} \frac{Rp \log 2}{2p\ceil{2R/\delta}} \\
&=  \lim_{\delta\to\infty} \lim_{R\to\infty} \frac{R\log 2}{2\ceil{2R/\delta}} \\
&=  \lim_{\delta\to\infty} \frac{\delta \log 2}{4} = \infty \qedhere
\end{align*}
\end{proof}

\begin{longRem}
Recall that the space from \cref{unboundedZero} is not coarsely equivalent to a space with bounded geometry and note that its volume growth is at least linear because it contains $\NN$. That is, from the perspective of having bounded geometry and from the perspective of volume growth, the space from \cref{unboundedZero} should be considered at least as large as the space from \cref{second-example}. However, the coarse entropy vanishes for the space from \cref{unboundedZero} and is infinite for the space from \cref{second-example}, so from the perspective of coarse entropy it is the latter that is larger. In particular, we can easily conclude that there is no coarse embedding of the latter in the former (see \cref{coarselyMonotone}).
\end{longRem}

Similarly, one can reprove many coarse non-embeddability results using coarse entropy. For example, the hyperbolic $n$-spaces $\HH^n$ (see \cref{hyperbolic}) for $n\geq 2$ do not admit a coarse embedding into the space from \cref{strangeMetricsExample} despite the fact that all these spaces have exponential volume growth.

\section{Theorems}\label{sec:theorems}

\subsection{Bounded geometry}\label{sec:boundedGeometry}

Now we investigate the connections between the coarse entropy of a
space and the ``sizes'' of balls in this space.

\begin{defi}\label{vol}
If $(X,d,\mu)$ is a metric measure space (i.e.\ a metric space $(X,d)$ with a Borel measure~$\mu$) and $x\in X$, define $\vol_\delta^x\colon \NN \to [0,\infty]$ by 
\[\vol_\delta^x(l) \defeq \sup_{x_0\in [x]_\delta} \mu\left(B_\delta(x_0,l)\right).\]
In particular, if $\mu$ is the counting measure on $X$, we will denote $\vol_\delta^x$ by $V_\delta^x$.
The superscript in $V_\delta^x$ will often be omitted when $x$ is fixed.
\end{defi}

When we use the notation $\mu(B_\delta(x_0,l))$, we assume that
$B_\delta(x_0,l)$ is a Borel set. While this is the case in all
natural situations, we do not know whether this is true in general.
Indeed, for a general metric space $(X,d)$, even if a set $A\subseteq X$
is Borel, the set $\bigcup_{x\in A}B(x,\delta)$ may be not Borel (private communication 
from L.~Snoha). However, this is not a problem
for us, since the closure of $B_\delta(x_0,l)$ is contained in
$B_\delta(x_0,l+1)$, and the formulas we use involve the limit as $l$
goes to infinity (so the measure of $B_\delta(x_0,l)$ can be replaced
by the measure of its closure if necessary).

The notation of \cref{vol} comes from the fact that
$\vol_\delta^x$ resembles the usual volume growth function $r\mapsto \mu(B(x_0,r))$.

Note that by the triangle inequality $B_\delta(x_0, n) \subseteq B(x_0, n\delta)$, so $V_\delta^x$ takes only finite values when $X$ has bounded geometry. Note further that for $\delta \in \NN$ the above inclusion is an equality for example when $X$ is a graph, which enables a simpler formulation of \cref{main} in this case, namely \cref{corollary:graph}. 

\begin{thm}\label{main-volume} Let $(X,d,\mu)$ be a metric measure space. Fix a basepoint $x\in X$. If $\vol_\delta^x(l)$ is finite for every $\delta > 0$ and $l \in \NN$, and there exists $\delta_0 > 0$ such that
\[\limsup_{l\to \infty} \fdfrac{1}{l} \log \vol_{\delta_0}^x(l) > 0,\]
then the coarse entropy is infinite.
\end{thm}

\begin{proof}
By the assumption, there is $\delta_0 > 0$ such that $\limsup_{l\to\infty} \frac{1}{l} \log V_{\delta_0}(l)$ equals $\log C_0$ for some $C_0 \in (1,\infty]$ (where $\log(\infty) = \infty$). Fix $C_1, C > 1$ such that  $C < C_1 < C_0$. By the assumption, there are sequences $(x_k)$ in $[x]_{\delta_0}$ and $(l_k)$ in $\NN$ such that $\mu\left(B_{\delta_0}(x_k, l_k)\right) > C_1^{l_k}$. Since $\vol_\delta^x(l)$ is always finite, for every $\delta \geq \delta_0$ the measure of balls $B(y, \delta)$ of radius~$\delta$ is bounded uniformly in $y\in [x]_{\delta_0}$, namely by $\vol_\delta^x(1)$. Since $C < C_1$, by taking a subsequence we can assume that 
\[\mu \left( B_{\delta_0}(x_k, l_k)\right) / \vol_k^x(1) \geq C^{l_k}.\]

Fix $\delta \geq \delta_0$, and denote $m \defeq \floor{\delta/\delta_0}$. Observe that for any $r\in \NN$
\begin{equation}\label{fasterHops}
B_{\delta_0}(x_k, r) \subseteq B_{\delta}\left(x_k,\ceil{r/m}\right).
\end{equation}

Now fix $R\in \NN$ with $R\geq \delta_0$. Consider the set $B_{\delta_0}(x_R, l_R)$ and pick a maximal $R$-separated subset $A$ of it. By maximality, $\bigcup_{a\in A} B(a, R)$ contains $B_{\delta_0}(x_R, l_R)$. Hence
\[|A| \cdot \vol_R^x(1) \geq \mu\left(\textstyle\bigcup_{a\in A} B(a, R)\right) \geq \mu\left( B_{\delta_0}(x_R, l_R) \right) \geq C^{l_R} \cdot \vol_R^x(1),\]
and thus the cardinality of $A$ is at least $C^{l_R}$.

By \eqref{fasterHops} and the definition of $B_{\delta}(x_R,\ceil{l_R/m})$, for every $a \in A$ there is a $\delta$-path~$o_a$ of length $\ceil{l_R/m}$ connecting $x_R$ and $a$. Since $A$ is $R$-separated, so is the set $\mathcal{O} = \{ o_a : a\in A \}$ of these paths, and clearly its cardinality equals the cardinality of~$A$.

By the definition of $[x]_{\delta_0}$ and the fact that $\delta \geq \delta_0$, there is a $\delta$-path $o$ connecting~$x$ and~$x_R$ of some finite length $L_R$. Now let $p\in \NN$ be arbitrary and consider the following set of $\delta$-paths starting at $x$:
\[\mathcal{O}^{(p)} = \big\{o * o_{a_1} * o_{a_1}^{-1} * \ldots * o_{a_p} * o_{a_p}^{-1} : \forall i \in\{1,\ldots, p\} \; a_i \in A \big\}.\]
By construction, the cardinality of $\mathcal{O}^{(p)}$ equals $|A|^p \geq C^{pl_R}$, and it consists of $\delta$-paths of length $n_p \defeq L_R + 2p \cdot \ceil{l_R/m}$. We have
\begin{align*}
\limsup_{p\to\infty} \fdfrac{1}{n_p} \log|\mathcal{O}^{(p)}| &
\geq
\lim_{p\to\infty} \frac{pl_R \log C}{L_R + 2p \ceil{l_R/m}}
= \frac{l_R \log C}{2 \ceil{l_R/m}}.
\end{align*}
Since in \eqref{spaceEntropySeparated} we are interested in $\limsup$ as $n\to \infty$, it suffices to obtain a lower bound on a sequence $(n_p)$ as above. Thus, by recalling that $m=\floor{\delta/\delta_0}$ and unfixing $\delta$ and~$R$, we conclude that
\[h_\infty(X) \geq \lim_{\delta\to \infty} \lim_{R \to \infty} \frac{l_R \log C}{2 \ceil{l_R/m}} = \lim_{\delta\to \infty} \frac{m\log C}{2} = \infty. \qedhere\]
\end{proof}

\begin{longEx}\label{hyperbolic}
 Consider the hyperbolic $n$-space $\HH^n$ for any $n \geq 2$. For the Riemannian metric on $\HH^n$, we have $B_\delta(x_0, l) = B(x_0, \delta l)$ for all $\delta > 0$, $l\in \NN$, and $x_0\in \HH^n$, and the Riemannian volume of $B(x_0, \delta l)$ does not depend on $x_0$ by homogeneity. It is well known that this volume is proportional to $\int_0^{\delta l} \sinh^{n-1}r\,dr$, so in particular $\vol_\delta^x(l)$ is finite for all $\delta > 0$ and $l\in \NN$, and $\lim_{l\to \infty} \frac{1}{l} \log \vol_\delta^x(l) = \log(e^{(n-1)\delta}) > 0$ for all $x\in \HH^n$. Hence, we can apply \cref{main-volume} to conclude that $h_\infty(\HH^n) = \infty$.
\end{longEx}

Out main result in this section is the following \cref{main}. Item~\labelcref{mainSecondAssertion} in \cref{main} is a special case of \cref{main-volume}.

\begin{thm}\label{main} Let $(X,d)$ be a metric space. Assume that $X$ has  bounded geometry. Fix a basepoint $x\in X$ and put $V_\delta = V_\delta^x$. We have the following dichotomy:
\begin{enumerate}[(i)]
\item\label{mainFirstAssertion} if for all $\delta > 0$
\begin{equation*}
\lim_{l\to \infty} \fdfrac{1}{l} \log V_\delta(l) = 0,
\end{equation*}
then the coarse entropy of $X$ is zero,
\item\label{mainSecondAssertion} if there exists $\delta > 0$ such that
\[\limsup_{l\to \infty} \fdfrac{1}{l} \log V_\delta(l) > 0,\]
then the coarse entropy of $X$ is infinite.
\end{enumerate}
\end{thm}

\begin{proof}
We begin with the proof of item~\labelcref{mainFirstAssertion}, for which 
we will use formula \eqref{spaceEntropyDense}.
Fix $\delta > 0$. Without loss of generality we can assume that $R > 0$ is of the form $R=(k+1)\delta$ for some
$k\in \NN$. Fix $n \in \NN$. For every $m\in\{1,2,\dots,k\}$ and every pair of points $y,y'\in B_\delta(x, n)$ such that there exists a $\delta$-path
from $y$ to $y'$ of length $m$, fix one such $\delta$-path, and denote it by
$p(y,y',m)$. Let $P$ be the set of the paths $p(y,y',m)$ over all such triples $(y,y',m)$. Observe
that if we fix $y$ and $m$, then the number of choices of~$y'$ equals $|B_\delta(y,m)| \leq |B_\delta(y,k)| \leq V_\delta(k)$.

We have $n=rk+s$, where $0\le s<k$ and $r = \floor{n/k}$. Let $F_n$ be the set of all
$\delta$-paths of length $n$ starting at $x$ that are concatenations of $\delta$-paths from $P$ with the following lengths: first, take $r$ $\delta$-paths of length $k$ and then one $\delta$-path of length $s$.

Then the set $F_n$ is $R$-dense in
the set of all $\delta$-paths of length $n$ originating at $x$.
Indeed, if $(x_i)$ is a $\delta$-path originating at $x$, then there exists a $\delta$-path $p=(p_0,\ldots, p_n)$ in $F_n$ such that $p_i = x_i$ for all $i \in \{0, k, \ldots, rk, n\}$. Thus for any other index $j\in \{0, \ldots, n\}$ there exists $i$ as above such that $|j-i| \leq k/2$, and hence $d(x_j, x_i) \leq \delta k / 2$ and $d(p_j, p_i) \leq \delta k / 2$, so indeed $d(x_j, p_j) \leq \delta k < R$.

By construction, the cardinality of $F_n$ is at most $V_\delta(k)^{r+1}$. Therefore, the minimal cardinality $r(n,R,\delta,x)$ of an $R$-dense subset of $\delta$-paths of length $n$ is bounded by 
$V_\delta(k)^{\frac{n}{k}+1} \leq V_\delta(k)^{\frac{2n}{k+1}+1}=V_\delta(R/\delta)^{\frac{2n\delta}{R}+1}$.
Hence,
\[
\limsup_{n\to\infty}\fdfrac 1 n \log |F_n|
\leq
\limsup_{n\to\infty}\frac{\big(\frac{2n}{R/\delta}+1\big)\log V_\delta\left(R/\delta\right)}{n} \displaybreak[0]
=
\frac{2\log V_\delta\left(R /\delta\right)}{ R /\delta}
\]
and by the assumption the last expression goes to zero as $R$ tends to infinity, which  by \eqref{spaceEntropyDense} shows that the coarse entropy of $X$ is zero, and finishes the proof of item~\labelcref{mainFirstAssertion}.

\medskip Item~\labelcref{mainSecondAssertion} is a special case of \cref{main-volume}. To see that, it suffices to observe that for $\mu$ being the counting measure the functions $\vol_\delta^x$ and $V_\delta^x$ coincide, and that under the assumptions of \cref{main} the value of $\vol_\delta^x(l)$ is finite for every $\delta,l \in \NN$ because
\[\vol_\delta^x(l) = \sup_{x_0\in [x]_\delta} \mu(B_\delta(x, l)) \leq \sup_{x_0\in X} \mu(B(x,\delta \cdot l)) < \infty,\]
where the last inequality is the definition of bounded geometry.
\end{proof}

\begin{longRem}\label{boundedGeometryNotReallyNeeded}
Note that the proof of item~\labelcref{mainFirstAssertion} in \cref{main} does not use the assumption that $X$ has bounded geometry. However, it deduces from the vanishing of the limit that $V_\delta(k)$ is finite for every $\delta>0$ and $k\in \NN$, which is equivalent to bounded geometry in many cases of interest, so we preferred to make it an explicit common assumption for both items of \cref{main}.
\end{longRem}

The following provides an example of a space that comes under the purview of \cref{main-volume} but not \cref{main}.

\begin{longEx}\label{measure-example} Let the rooted tree $T$ be defined as follows. The root has 2 children, each child of the root has 3 children, each grand-child of the root has 4 children, and so on. 
\end{longEx}

Clearly, $T$ from \cref{measure-example} does not have bounded geometry itself, and by \cref{ceBoundedGeometry} it is easy to see that it is not coarsely equivalent to a space with bounded geometry. Hence, its coarse entropy cannot be calculated with \cref{main}.\ref{mainFirstAssertion}.

\begin{claim} The tree $T$ from \cref{measure-example} satisfies the assumptions of \cref{main-volume}.
\end{claim} 
\begin{proof}
Define a measure $\mu$ on $T$ by assigning to every vertex of $T$ at distance~$n$ from the root the measure $\frac{2^n}{(n+1)!}$. To see that the assumptions of \cref{main-volume} are satisfied, first note that the measure of every singleton is bounded by one. When $v$ is at distance~$n$ from the root and $C_v = \{v_1,\ldots, v_{n+2}\}$ is the set of children of $v$, then
\[2 \mu(\{v\}) =  2 \cdot \frac{2^n}{(n+1)!} = (n+2) \cdot \frac{2^{n+1}}{(n+2)!} = \mu(C_v).\]

It follows that for any vertex $v$ the subtree $T_v(l)$ consisting of descendants of $v$ at distance at most $l\in \NN \cup \{0\}$ from $v$ has measure $(2^{l+1}-1)\mu(v)$. In particular, if we take $v$ as the root, we see that
\[\limsup_{l\to \infty} \fdfrac{1}{l} \log \vol_{1}^v(l) \geq \lim_{l\to\infty}\fdfrac{1}{l} \log \mu(B(v,l)) = \lim_{l\to\infty} \fdfrac{1}{l} \log \mu(T_v(l)) = \log 2.\]

For every vertex $v$ that is not the root, denote by $p(v)$ its parent. Then, we see that $B(v,l) = \bigcup_{i=0}^{\min(l,n)} T_{p^i(v)}(l-i)$, where $n$ is the distance of $v$ from the root.
Consequently, we get
\[\mu(B(v,l)) \leq
\sum_{i=0}^{\min(l,n)} (2^{l-i+1}-1) \cdot \frac{2^{n-i}}{(n-i+1)!}
\leq 2^{l+2},
\]
and since for every $\delta > 0$ we have $B_\delta(v,l) = B(v, \floor{\delta} l)$, we conclude that $\vol_\delta^x(l)\leq 2^{\delta l + 2}$, so it is finite for every $\delta> 0$ and $l\in \NN$. This shows that the space $T$  satisfies the assumptions of \cref{main-volume}.
\end{proof}

The space $T$ from \cref{measure-example} is also covered by \cref{graphUnbounded} from the next section. However, \cref{graphUnbounded} requires the space to be quasigeodesic, which is not necessary for \cref{main-volume}.

\Cref{measure-example} suggests that it may be fruitful to generalise \cref{main}.\labelcref{mainFirstAssertion} using $\vol_\delta^x(l)$, analogous to \cref{main-volume} generalising \cref{main}.\labelcref{mainSecondAssertion}. However, such a generalisation
requires an additional assumption on the measure $\mu$: without any such additional assumption, for every separable $X$ we can arrange $\mu(X)$ to be finite, and then the naive analogues of the assumptions of \cref{main}.\labelcref{mainFirstAssertion} are trivially satisfied. In \cref{main-corollary}, we propose the following additional assumption: there exists $r > 0$ such that $\mu(B(x,r))$ is bounded away from zero uniformly in $x\in X$. Although \cref{main-corollary} is formally more general than \cref{main}.\labelcref{mainFirstAssertion}, we derive below \cref{main-corollary} from \cref{main}.\labelcref{mainFirstAssertion}.

\begin{thm}\label{main-corollary} Let $(X,d,\mu)$ be a metric measure space. Assume that there exists $r > 0$ such that $\inf_{x\in X} \mu(B(x,r)) > 0$. Fix a basepoint $x_0\in X$. If for all $\delta > 0$
\begin{equation}
\label{subexp}
\lim_{l\to \infty} \fdfrac{1}{l} \log \vol_{\delta}^{x_0}(l) = 0,
\end{equation}
then the coarse entropy is zero.
\end{thm}
\begin{proof}
Let $Z\subseteq X$ be a maximal $3r$-separated subset containing $x_0$. By \cref{coequiv}, it suffices to prove that the coarse entropy of $Z$ is zero. Let $\delta > 0$ and consider the ball $B_Z(x_0, \delta)$ in $Z$. Note that $\bigcup_{z\in B_Z(x_0, \delta)} B_X(z,r)$ is contained in the ball $B_X(x_0,\delta+r)$, whose measure is bounded by $\vol_{\delta+r}^{x_0}(1)$ (which is finite by \eqref{subexp}). Consequently,
\[|B_Z(x_0, \delta)| \leq \frac{\vol_{\delta+r}^{x_0}(1)}{\inf_{x\in X} \mu(B(x,r))}.\]
In particular, $B_Z(x_0, \delta)$ is finite for every $\delta > 0$, and hence $Z$ is countable. Let $p\colon X\to Z$ be any measurable closest-point retraction. (To construct it, first arrange the elements of $Z$ in a sequence $(z_1, z_2, \ldots)$. By the finiteness of balls, for every $x\in X$ there is a finite set $C_x\subseteq Z$ such that every $z\in C_x$ minimises the distance between~$x$ and $Z$ (in particular, the infimum of distances is acquired). Define $p(x)$ as $z\in C_x$ with the smallest index in the sequence $(z_1, z_2, \ldots)$.)

Consider the pushforward measure $p_*(\mu)$. By $3r$-disjointness of $Z$, the inverse image $p^{-1}(z)$ contains the ball $B(z,r)$ (in fact it contains the open ball of radius $3r/2$). Consequently, $p_*(\mu)(\{z\}) \geq \inf_{x\in X} \mu(B(x,r))$. We will also later use that by the $3r$-density of $Z$ the inverse image $p^{-1}(z)$ is contained in $B(z,3r)$.

For $z\in Z$, $\delta> 0$, and $l\in \NN$, denote by $[z]_{Z,\delta}$ the $\delta$-component of $z$ in $Z$, and by $B_{Z,\delta}(z,l)$ the $\delta$-ball in $Z$ of radius $l$ around $z$. Clearly, these are contained in the corresponding $\delta$-component and $\delta$-ball in $X$. Hence
\begin{align*}
\sup_{z\in [x_0]_{Z,\delta}} |B_{Z,\delta}(z, l)|
&\leq
\frac{\sup_{z\in [x_0]_{Z,\delta}} \big(p_*(\mu)\big)\big(B_{Z,\delta}(z, l)\big)}{\inf_{z\in Z} p_*(\mu)(\{z\})}\displaybreak[0]\\
&\leq
\frac{\sup_{x\in [x_0]_{\delta}} \mu\big(B_{\delta+3r}(x, l)\big)}{\inf_{x\in X} \mu(B(x,r))}\displaybreak[0]\\
&\leq
\frac{\vol^{x_0}_{\delta + 3r}(l)}{\inf_{x\in X} \mu(B(x,r))},
\end{align*}
and thus the assumption on $X$ that $\lim_{l\to \infty} \frac{1}{l} \log \vol_{\delta}^{x_0}(l)$ vanishes for all $\delta$ implies the main and only (in the light of \cref{boundedGeometryNotReallyNeeded}) assumption of item~\labelcref{mainFirstAssertion} in \cref{main} for~$Z$, namely that $\lim_{l\to \infty} \frac{1}{l} \log V_\delta^{x_0}(l)$ vanishes for all~$\delta$. Hence, the coarse entropy of~$Z$ vanishes. 
\end{proof}

Recall from \cref{qg} that any quasigeodesic metric space (in particular any geodesic space, like a Riemannian manifold, for instance the hyperbolic plane) is coarsely equivalent to (even quasi-isometric to) a connected graph. In these cases, \cref{main} boils down to the following.
\begin{thm}\label{corollary:graph} Let $X$ be a quasigeodesic metric space of bounded geometry, for example a connected graph with the degrees of vertices uniformly bounded. We have the following dichotomy:
\begin{enumerate}[(i)]
\item\label{first-graph} if $\lim_{l\to\infty} \frac{1}{l} \log \left(\sup_{x\in X}|B(x,l)|\right) = 0$, then the coarse entropy of $X$ is zero,
\item\label{second-graph} if $\limsup_{l\to\infty} \frac{1}{l} \log \left(\sup_{x\in X}|B(x,l)|\right) > 0$, then the coarse entropy is infinite.
\end{enumerate}
\end{thm}
\begin{proof}
Note that if $X$ is a graph such that the degree of vertices is bounded by $d\in \NN$, then the cardinality of any ball of radius $n\in \NN$ is at most $1+d\sum_{i=1}^n(d-1)^{i-1}$ (this upper bound is attained for the infinite $d$-regular tree), in particular $X$ has bounded geometry.

Now, if $X$ is quasigeodesic, by \cref{qg} there exists $\delta > 0$ such that any two points within distance~$n$ from each other can be connected by a $\delta$-path of length~$n$. In particular, the $\delta$-component $[x]_\delta$ of any $x\in X$ is the whole of $X$. 

\smallskip\noindent\emph{Item~\labelcref{second-graph}.~~~} Consequently, for any $\delta'\geq \max(\delta,1)$ and $l\in \NN$ we have
\[V_{\delta'}^x(l) = \sup_{x_0\in [x]_{\delta'}} |B_{\delta'}(x_0, l)| =
\sup_{x_0\in X} |B_{\delta'}(x_0, l)| \geq \sup_{x_0\in X} |B(x_0, l)|.\]
Hence, if $\limsup_{l\to\infty} \frac{1}{l} \log \left(\sup_{x\in X}|B(x,l)|\right) > 0$, then the assumption from \cref{main}.\labelcref{mainSecondAssertion} holds, and thus the coarse entropy is infinite.

\smallskip\noindent\emph{Item~\labelcref{first-graph}.~~~} Now, for an arbitrary $\delta > 0$, using the inclusion $B_\delta(x,l) \subseteq B(x,l\delta)$, we obtain
\[V_{\delta}^x(l) = \sup_{x_0\in [x]_\delta} |B_\delta(x_0, l)| \leq
\sup_{x_0\in X} |B_\delta(x_0, l)| \leq \sup_{x_0\in X} |B(x_0, \delta l)|,\]
so 
\begin{align*}
\limsup_{l\to \infty} \fdfrac{1}{l} \log V_{\delta}^x(l)
&\leq
\limsup_{l\to \infty} \fdfrac{1}{l} \log \left(\sup_{x_0\in X} |B(x_0, \delta l)|\right)
\displaybreak[0]\\&=
\delta \limsup_{l\to \infty} \fdfrac{1}{\delta l} \log \left(\sup_{x_0\in X} |B(x_0, \delta l)|\right)
\displaybreak[0]\\&\leq
\delta \limsup_{L\to \infty} \fdfrac{1}{L} \log \left(\sup_{x_0\in X} |B(x_0, L)|\right).
\end{align*}
Hence, if $\lim_{l\to\infty} \frac{1}{l} \log \left(\sup_{x\in X}|B(x,l)|\right) = 0$, then the assumption from \cref{main}.\labelcref{mainFirstAssertion} holds, and thus the coarse entropy vanishes.
\end{proof}

For a connected vertex-transitive graph of finite degree—e.g.\ the Cayley graph of a finitely generated group with respect to a finite generating set—the dichotomy from \cref{corollary:graph} distinguishes between subexponential and exponential growth.
\begin{cor}\label{vertex-transitive} Let $X$ be a connected vertex-transitive graph with vertices of finite degree. Fix any vertex $x\in X$. We have the following dichotomy:
\begin{enumerate}[(i)]
\item if $\lim_{l\to\infty} \frac{1}{l} \log \left(|B(x,l)|\right) = 0$, then the coarse entropy of $X$ is zero,
\item if $\limsup_{l\to\infty} \frac{1}{l} \log \left(|B(x,l)|\right) > 0$, then the coarse entropy is infinite.
\end{enumerate}
\end{cor}

\begin{cor}\label{groups}
A finitely generated group has exponential growth if and only if it
has infinite coarse entropy as a metric space with the word length distance.
\end{cor}

Results parallel to some of those in this section for other entropy notions have appeared in \cites{DikranjanGiordanoBruno,Zava}.

\subsection{Unbounded geometry}\label{sec:unboundedGeometry}

In this section, we consider the less tame case of metric spaces without bounded geometry. One expects that such spaces would have infinite coarse entropy, which is indeed usually true as shown in \cref{graphUnbounded}. Together with the results of the previous section, this yields a complete understanding (\cref{completeQuasigeodesic}) of coarse entropy for quasigeodesic spaces.

Interestingly, in general there exist spaces without bounded geometry whose coarse entropy vanishes, and in \cref{unboundedZero} we even provided a 1-connected example of this kind. This shows that the assumption that $X$ is quasi-geodesic in \cref{graphUnbounded} cannot be relaxed to the assumption that $X$ is $C$-connected for some $C>0$.

\begin{thm}\label{graphUnbounded}
Let $X$ be a quasigeodesic space, e.g.\ a connected graph, and assume that it is not coarsely equivalent to a metric space of bounded geometry. Then $h_\infty(X)=\infty$.
\end{thm}

\begin{proof}
(Compare the proof of \cref{main-volume} and \cref{measure-example}.)
By \cref{ceBoundedGeometry}, for every $s > 0$ there is $D_s \in \NN$ and a sequence $(A_k^s)_{k=1}^\infty$ of $s$-separated sets in $X$ of diameter at most $D_s$ such that $\lim_k |A_k^s| = \infty$. For every~$A_k^s$ pick a basepoint $x_k^s\in A_k^s$.

Fix a ``global'' basepoint $x_0\in X$, $R > 0$, and $\delta > 0$ sufficiently large that item~\ref{qg2} of \cref{qg} holds. Let $k\in \NN$.  Let $o$ be any $\delta$-path between $x_0$ and $x_k^R$ and let~$l$ denote its length. For every $a\in A_k^R$ there is a $\delta$-path $o_a$ connecting $x_k^R$ and $a$ with length at most ${D_R}$, and by repeating its entries if necessary it can be lengthened to a $\delta$-path $o_a$ of length exactly $D_R$ connecting $x_k^R$ and $a$ . For any $p\in \NN$ consider the following set of $\delta$-paths:
\[\mathcal{O}^{(p)} = \big\{o * o_{a_1} * o_{a_1}^{-1} * \ldots * o_{a_p} * o_{a_p}^{-1} : \forall i \in\{1,\ldots, p\} \; a_i \in A_k^R \big\}.\]
The cardinality of $\mathcal{O}^{(p)}$ equals $|A_k^R|^p$, it consists of paths of length $l+2p {D_R}$, and it is $R$-separated because $A_k^R$ is $R$-separated. Thus we get
\begin{align*}
h_\infty(X) &\geq
\lim_{\delta\to \infty} \lim_{R \to \infty} \limsup_{p\to \infty} \frac{\log |\mathcal{O}^{(p)}|}{l+2p {D_R}}
\\&=
\lim_{\delta\to \infty} \lim_{R \to \infty} \limsup_{p\to \infty} \frac{p \log |A_k^R|}{l+2p {D_R}} 
=
\lim_{\delta\to \infty} \lim_{R \to \infty} \frac{\log |A_k^R|}{2 {D_R}},
\end{align*}
and since $k$ was arbitrary and $\lim_{k\to\infty} |A_k^s| = \infty$ for every $s\in \NN$, we conclude
\[
h_\infty(X) \geq
\lim_{\delta\to \infty} \lim_{R \to \infty} \lim_{k\to \infty} \frac{\log |A_k^R|}{2 {D_R}} = \infty.\qedhere
\]
\end{proof}

\Cref{graphUnbounded} together with \cref{main} yields a complete classification for quasigeodesic spaces.

\begin{cor}\label{completeQuasigeodesic}
Let $X$ be a quasigeodesic space, e.g.\ a connected graph.
\begin{enumerate}[(a)]
\item If $X$ is not coarsely equivalent to a space with bounded geometry, then the coarse entropy of $X$ is infinite.
\item Otherwise, there exists a connected bounded-degree graph $G$ quasi-isometric to $X$, and the usual dichotomy holds:
\begin{enumerate}[(i)]
\item if $\lim_{l\to\infty} \frac{1}{l} \log \left(\sup_{x\in G}|B(x,l)|\right) = 0$, then the coarse entropy of $X$ is zero,
\item if $\limsup_{l\to\infty} \frac{1}{l} \log \left(\sup_{x\in G}|B(x,l)|\right) > 0$, then the coarse entropy of $X$ is infinite.
\end{enumerate}
\end{enumerate}
\end{cor}
\begin{proof}
If $X$ is not coarsely equivalent to a space of bounded geometry, then \cref{graphUnbounded} applies, so $h_\infty(X) = \infty$.

If $X$ is coarsely equivalent to a space of bounded geometry, then it follows from \cref{qg} that there exists a bounded degree graph that is quasi-isometric to~$X$. By \cref{coequiv}, the rest of the claim follows immediately from \cref{corollary:graph}.
\end{proof}

\section*{Data availability statement}

Data sharing not applicable to this article as no datasets were generated or analysed during the current study.


\begin{bibsection}
\begin{biblist}
\bib{DikranjanGiordanoBruno}{article}{
   author={Dikranjan, Dikran},
   author={Giordano Bruno, Anna},
   title={Topological entropy and algebraic entropy for group endomorphisms},
   book={
       title = {Proceedings ICTA2011},
       publisher = {Cambridge Scientific Publishers},
       date={2012}
   },
   conference={
       title={International Conference on Topology and its Applications ICTA 2011},
       address={Islamabad, Pakistan, July 2011}
   },
   pages={133--214},
}

\bib{GM}{article}{
   author={Geller, William},
   author={Misiurewicz, Micha\l },
   title={Coarse entropy},
   journal={Fund. Math.},
   volume={255},
   date={2021},
   number={1},
   pages={91--109},
   issn={0016-2736},
}

\bib{Guentner--Kaminker}{article}{
   author={Guentner, Erik},
   author={Kaminker, Jerome},
   title={Exactness and uniform embeddability of discrete groups},
   journal={J. London Math. Soc. (2)},
   volume={70},
   date={2004},
   number={3},
   pages={703--718},
   issn={0024-6107},
}

\bib{HMT}{article}{
   author={Hume, David},
   author={Mackay, John M.},
   author={Tessera, Romain},
   title={Poincar\'{e} profiles of groups and spaces},
   journal={Rev. Mat. Iberoam.},
   volume={36},
   date={2020},
   number={6},
   pages={1835--1886},
   issn={0213-2230},
}

\bib{LMR}{article}{
   author={Lubotzky, Alexander},
   author={Mozes, Shahar},
   author={Raghunathan, M. S.},
   title={The word and Riemannian metrics on lattices of semisimple groups},
   journal={Publ. Math. Inst. Hautes Études Sci.},
   volume={91},
   date={2000},
   pages={5--53},
   issn={0073-8301},
}

\bib{NY}{book}{
   author={Nowak, Piotr W.},
   author={Yu, Guoliang},
   title={Large scale geometry},
   series={EMS Textbooks in Mathematics},
   publisher={European Mathematical Society (EMS), Z\"{u}rich},
   date={2012},
   pages={xiv+189},
   isbn={978-3-03719-112-5},
   doi={10.4171/112},
}

\bib{thesis}{thesis}{
  author = {Sawicki, Damian},
  title = {On the geometry of metric spaces defined by group actions: from circle rotations to super-expanders},
  school = {Institute of Mathematics of the Polish Academy of Sciences},
  date = {2018},
  type = {PhD thesis},
}

\bib{Smith}{article}{
   author={Smith, J.},
   title={On asymptotic dimension of countable Abelian groups},
   journal={Topol. Appl.},
   volume={153},
   date={2006},
   number={12},
   pages={2047--2054},
}

\bib{Zava}{article}{
   author={Zava, Nicol\`o},
   title={On a notion of entropy in coarse geometry},
   journal={Topol. Algebra Appl.},
   volume={7},
   date={2019},
   number={1},
   pages={48--68},
}

\end{biblist}
\end{bibsection}
\end{document}